\documentclass[a4paper]{article}
\usepackage{amsmath,amsthm, authblk}
\usepackage{amssymb,latexsym}
\usepackage{graphicx,subfig}
\usepackage{epsfig}
\usepackage{color}

\usepackage{lineno}
\newcommand*\patchAmsMathEnvironmentForLineno[1]{%
  \expandafter\let\csname old#1\expandafter\endcsname\csname #1\endcsname
  \expandafter\let\csname oldend#1\expandafter\endcsname\csname end#1\endcsname
  \renewenvironment{#1}%
     {\linenomath\csname old#1\endcsname}%
     {\csname oldend#1\endcsname\endlinenomath}}%
\newcommand*\patchBothAmsMathEnvironmentsForLineno[1]{%
  \patchAmsMathEnvironmentForLineno{#1}%
  \patchAmsMathEnvironmentForLineno{#1*}}%
\AtBeginDocument{%
\patchBothAmsMathEnvironmentsForLineno{equation}%
\patchBothAmsMathEnvironmentsForLineno{align}%
\patchBothAmsMathEnvironmentsForLineno{flalign}%
\patchBothAmsMathEnvironmentsForLineno{alignat}%
\patchBothAmsMathEnvironmentsForLineno{gather}%
\patchBothAmsMathEnvironmentsForLineno{multline}%
}

\newtheorem{thm}{Theorem}[section]
\newtheorem{dfn}[thm]{Definition}
\newtheorem{lma}[thm]{Lemma}
\newtheorem{cor}[thm]{Corollary}
\newtheorem{prp}[thm]{Proposition}

\newtheorem{clm}[thm]{Claim}

\newtheorem{conj}[thm]{Conjecture}

\title{A multipartite version of the Hajnal--Szemer\'{e}di theorem for graphs and hypergraphs
}
\author{Allan Lo
\footnote{This author was supported by the ERC, grant no.~258345.}
}
\affil{School of Mathematics, University of Birmingham, \\Birmingham, B15 2TT, UK\\
\texttt{s.a.lo@bham.ac.uk}}
\author{Klas Markstr{\"o}m}
\affil{Department of Mathematics and Mathematical Statistics,\\ Ume\r{a} University, S-901 87 Ume\r{a}, Sweden\\
\texttt{klas.markstrom@math.umu.se}}

\begin{document}


\maketitle 

\abstract{
A perfect $K_t$-matching in a graph $G$ is a spanning subgraph consisting of vertex disjoint copies of~$K_t$.
A classic theorem of Hajnal and Szemer\'{e}di states that if $G$ is a graph of order $n$ with minimum degree $\delta(G) \ge (t-1)n/t$ and $t| n$, then $G$ contains a perfect $K_t$-matching.
Let $G$ be a $t$-partite graph with vertex classes $V_1$, \dots, $V_t$ each of size~$n$.
We show that for any $\gamma >0$ if every vertex $x \in V_i$ is joined to at least $\left((t-1)/t + \gamma\right)n $ vertices of $V_j$ for each $j \ne i$, then $G$ contains a perfect $K_t$-matching, provided $n$ is large enough.
Thus, we verify a conjecture of Fisher~\cite{MR1698745} asymptotically.
Furthermore, we consider a generalisation to hypergraphs in terms of the codegree.
}

\section{Introduction}

Given a graph $G$ and an integer $t \ge 3$, a \emph{$K_t$-matching} is a set of vertex disjoint copies of~$K_t$ in~$G$.
A \emph{perfect $K_t$-matching} (or \emph{$K_t$-factor}) is a spanning $K_t$-matching.
Clearly, if $G$ contains a perfect $K_t$-matching then $t$ divides~$|G|$.
A classic theorem of Hajnal and Szemer\'{e}di~\cite{MR0297607} states a relationship between the minimum degree and the existence of a perfect $K_t$-matching.
\begin{thm}[Hajnal--Szemer\'{e}di Theorem~\cite{MR0297607}]
Let $t>2$ be an integer.
Let $G$ be a graph of order $n$ with minimum degree $\delta(G) \ge (t-1)n/t$ and $t| n$.
Then $G$ contains a perfect $K_t$-matching.
\end{thm}

Let $G$ be a $t$-partite graph with vertex classes $V_1$, \dots, $V_t$.
We say that $G$ is \emph{balanced} if $|V_i| = |V_j|$ for $1 \le i < j \le t$.
Write $G[V_i,V_j]$ for the induced bipartite subgraph on vertex classes $V_i$ and $V_j$.
Define $\widetilde{\delta}(G)$ to be $\min_{1 \le i < j \le t} \delta(G[V_i,V_j])$.
Fischer~\cite{MR1698745} conjectured the following multipartite version of the Hajnal--Szemer\'{e}di theorem.

\begin{conj}[Fischer~\cite{MR1698745}] \label{conj:hajnalszemeredi}
Let $G$ be a balanced $t$-partite graph with each class of size~$n$.
Then there exists an integer $a_{n,t}$ such that if $\widetilde{\delta}(G) \ge (t-1)n/t + a_{n,t}$, then $G$ contains a perfect $K_t$-matching.
\end{conj}

Note that the $+ a_{n,t}$ term was not presented in Fischer's original conjecture, but it was shown to be necessary for odd~$t$ in~\cite{MR1910115}.
For $t=2$, the conjecture can be easily verified by Hall's Theorem.
For $t=3$, Johansson~\cite{MR1735337} proved that $\widetilde{\delta}(G) \ge 2n/3 +\sqrt{n}$ suffices for all~$n$.
Using the regularity lemma, Magyar and Martin~\cite{MR1910115}, and Martin and Szemer{\'e}di~\cite{MR2433861} proved Conjecture~\ref{conj:hajnalszemeredi} for $t=3$ and $t=4$ respectively for $n$ sufficiently large, where $a_{n,t} = 1$ if both $t$ and $n$ are odd, $a_{n,t}=0$ otherwise.
For $t \ge 5$, Csaba and Mydlarz~\cite{csaba2008approximate} proved that $\widetilde{\delta}(G) \ge c_t n/ (c_t+1)$ is sufficient, where $c_t = t - 3/2 + (1+ 1/2 + \dots + 1/t)/2$.
In this paper, we show that Conjecture~\ref{conj:hajnalszemeredi} is true asymptotically.

\begin{thm} \label{thm:hajnalszemeredi}
Let $t \ge 2$ be an integer and let $\gamma> 0$.
Then there exists an integer $n_0 = n_0(t,\gamma)$ such that if $G$ is a balanced $t$-partite graph with each class of size~$n \ge n_0 $ and $\widetilde{\delta}(G) \ge \left( (t-1)/t+\gamma \right)n $, then $G$ contains a perfect $K_t$-matching.
\end{thm}

Independently, Theorem~\ref{thm:hajnalszemeredi} also has been proved by Keevash and Mycroft~\cite{keevash2011geometric}.
Their proof involves the hypergraph blowup lemma~\cite{keevashhypergraph}, so $n_0$ is extremely large, whereas our proof gives a much smaller $n_0$.
Since the submission of this paper, Keevash and Mycroft~\cite{keevash2012hajnal} have proved Conjecture~\ref{conj:hajnalszemeredi}, provided $n$ is large enough.
Also, Han and Zhao~\cite{HanZhao2012hajnal} gave a different proof of Conjecture~\ref{conj:hajnalszemeredi} for $t = 3,4$, again provided $n$ is large enough.

We further generalise Theorem~\ref{thm:hajnalszemeredi} to hypergraphs. 
For $a \in \mathbb{N}$, we refer to the set $\{1, \dots, a\}$ as $[a]$.
For a set $U$, we denote by $\binom{U}{k}$ the set of $k$-sets of~$U$.
A \emph{$k$-uniform hypergraph}, or \emph{$k$-graph} for short, is a pair $H = (V(H),E(H))$, where $V(H)$ is a finite set of vertices and $E(H) \subset \binom{V(H)}{k}$ is a family of $k$-sets of~$V(H)$.
We simply write $V$ to mean $V(H)$ if it is clear from the context.
For a $k$-graph $H$ and an $l$-set $T \in \binom{V}l$, let $N^H(T)$ be the set of $(k-l)$-sets $S \in \binom{V}{k-l}$ such that $S \cup T$ is an edge in $H$.
Let $\deg^H(T) = |N^H(T)|$.
Define the \emph{minimum $l$-degree} $\delta_l(H)$ of $H$ to be the minimal $\deg^H(T)$ over all $T \in \binom{V}l$.
For $U \subset V$, we denote by $H[U]$ the induced subgraph of $H$ on vertex set~$U$.

A $k$-graph $H$ is \emph{$t$-partite}, if there exists a partition of the vertex set $V$ into $t$ classes $V_1$, \dots, $V_t$ such that every edge intersects every class in at most one vertex.
Similarly, $H$ is \emph{balanced} if $|V_1| = \dots =|V_t|$.
An $l$-set $T \in \binom{V}l$ is said to be \emph{legal} if $| T \cap V_i| \le 1$ for $i \in [t]$.
For $I \subset [t]$, $T \subset V$ is \emph{$I$-legal} if $| T \cap V_i| = 1$ for $i \in I$ and $| T \cap V_i| = 0$ otherwise.
We write $V_I$ to be the set of $I$-legal sets.
For disjoint sets $I, J$ such that $I \cup J \in \binom{[t]}{k}$ and an $I$-legal set $T\in V_I$, denote by $N^H_J(T)$ the set of $J$-legal sets~$S$ such that $S \cup T$ is an edge in~$H$ and write $\deg^H_J(T) = |N^H_J(T)|$.
For $l \in [k-1]$ and $I \in \binom{[t]}{l}$, define $\widetilde{\delta}_I(H) = \min\{\deg^H_J(T) : T \in V_I $ and $ J \in \binom{[t] \backslash I}{ k-|I|}  \}$.
Finally, we set $\widetilde{\delta}_l(H) = \min\{\widetilde{\delta}_I(H) : I \in \binom{[t]}{l} \}$.
If $H$ is clear from the context, we drop the superscript of~$H$.
Note that for graphs, when $k=2$, $\widetilde{\delta}_1(G) = \widetilde{\delta}(G)$ as defined eariler.

Let $K_t^{k}$ be the complete $k$-graph on $t$ vertices.
It is easy to see that a $t$-partite $k$-graph $H$ contains a perfect $K^k_t$-matching only if $H$ is balanced.

\begin{dfn}
Let $1 \le l <k\le t$ and $n\ge 1$ be integers.
Define $\phi_l^k(t,n)$ to be the smallest integer $d$ such that every $t$-partite $k$-graph $H$ with each class of size $n$ and $\widetilde{\delta}_l(H) \ge d$ contains a perfect $K^k_t$-matching.
Equivalently, 
\begin{align}
\phi_l^k(t,n) = \min \{ d :  \textrm{ $\widetilde{\delta}_l(H) \ge d$} \Rightarrow \textrm{$H$ contains a perfect $K_t^{k}$-matching}\}, \nonumber
\end{align}
where $H$ is a $t$-partite $k$-graph $H$ with each class of size~$n$.
Write $\phi^k(t,n)$ for $\phi_{k-1}^k(t,n)$.
\end{dfn}

Note that Theorem~\ref{thm:hajnalszemeredi} implies that $\phi^2(t,n) \sim (t-1)n/t$.
Various cases of $\phi_l^k(k,n)$ have been studied.
Daykin and H{\"a}ggkvist~\cite{MR615135} showed that $\phi^k_1(k,n)  \le (k-1)n^{k-1}/k$, which was later improved by H\'{a}n, Person and Schacht~\cite{MR2496914}.
K{\"u}hn and Osthus~\cite{MR2207573} showed that $n/2-1< \phi^k(k,n) = \phi_{k-1}^k(k,n) \le n/2 +  \sqrt{2n \log n}$.
Aharoni, Georgakopoulos and Spr{\"u}ssel~\cite{MR2460215} then reduced the upper bound to $\phi^k(k,n)  \le \lceil (n+1)/2 \rceil$.
For $k/2 \le l <k-1$, Pikhurko~\cite{MR2438870} showed that $\phi_{l}^k(k,n) \le n^{k-l}/2$.
The exact value of $\phi^3_1(3,n)$ has been determined by the authors in~\cite{lo2011perfect}.
In this paper, we give an upper bound on $\phi^k(t,n)$ for $3 \le k<t$.

\begin{thm} \label{thm:partite}
For $3 \le k < t$ and $\gamma \ge 0$, there exists an integer $n_0 = n_0(k,t,\gamma)$ such that for all  $n\ge n_0$
\begin{align}
\phi^k(t,n) \le \left( 1 - \left( \binom{t-1}{k-1} + 2 \binom{t-2}{k-2} \right)^{-1} + \gamma\right)n. \nonumber
\end{align}
\end{thm}
We do not believe the upper bound is best possible.
For $k=3$ and $t=4$, it was shown, independently in~\cite{lo2011f} and~\cite{keevash2011geometric}, that for any $\gamma >0$ if $H$ is a 3-graph (not 3-partite) with $\delta_2(H) = (3/4+ \gamma) n$, then $H$ contains a perfect $K_4^3$-matching, provided $n$ is large enough.
(Moreover, in~\cite{keevash2011geometric}, Keevash and Mycroft have determined the exact value of $\delta_2(H)$-threshold for the existence of perfect $K_4^3$-matchings.)
Thus, it is natural to believe that $\phi^3(4,n)$ should be $3n/4 +o(n)$.

Our proofs of Theorem~\ref{thm:hajnalszemeredi} and Theorem~\ref{thm:partite} use the absorption technique introduced by R\"odl, Ruci\'{n}ski and Szemer\'{e}di~\cite{MR2500161}.
We now present an outline of the absorption technique.
First, we remove a set $U$ of disjoint copies of $K_t^k$ from $H$ satisfying the conditions of the absorption lemma, Lemma~\ref{lma:absorptionlemma}, and call the resulting graph $H'$.
Next, we find a $K_t^k$-matching covering almost all vertices of $H'$.
Let $W$ be the set of `leftover' vertices.
By the absorption property of $U$, there is a perfect $K^k_t$-matching in $H[U \cup W]$.
Hence, we obtain a perfect $K^k_t$-matching in $H$ as required.

In order to find a $K_t^k$-matching covering almost all vertices of $H'$, we follow the approach of Alon, Frankl, Huang, R\"{o}dl, Ruci\'{n}ski and Sudakov~\cite{alon2011large}, who consider fractional matchings. 
Let $\mathcal{K}_t^k(H)$ be the set of $K_t^k$ in a $k$-graph $H$.
A \emph{fractional $K_t^k$-matching} in a $k$-graph $H$ is a function $w: \mathcal{K}_t^k(H) \rightarrow [0,1]$ such that for each $v \in V$ we have 
\begin{align*}
\sum \{ w(T) : v \in T \in \mathcal{K}_t^k(H)  \}\le 1.
\end{align*}
Then $\sum_{T\in \mathcal{K}_t^k(H)} w(T)$ is the \emph{size} of $w$.
If the size is $|H|/t$, then $w$ is \emph{perfect}.
We are interested in perfect fractional $K_t^k$-matchings $w$ in a $t$-partite $k$-graph $H$ with each class of size~$n$.
Note that $|H| = tn$, so if $w$ is a perfect fractional $K_t^k$-matching in $H$, then
\begin{align*}
\sum \{ w(T) : v \in T \in \mathcal{K}_t^k(H)  \} = 1 \textrm{ for $v \in V$ and } \sum_{T\in \mathcal{K}_t^k(H)} w(T) = n.
\end{align*}
Define $\phi^{\ast,k}_l(t,n)$ to be the fractional analogue of $\phi^k_l(t,n)$.

\begin{thm} \label{thm:fractionalKq}
For $2 \le k \le t$ and $n \ge 1$,
\begin{align}
\lceil(t-k+1)n/t \rceil \le \phi^{\ast,k}(t,n) \le 
\begin{cases}
\lceil(t-1)n/t \rceil	& \textrm{for $k=2$,}\\
\left\lceil \left( 1-\binom{t-1}{k-1}^{-1}\right) n \right\rceil  +1 & \textrm{for $k\ge 3.$}
\end{cases} \nonumber
\end{align}
In particular, $\phi^{\ast,2}(t,n) = \lceil(t-1)n/t \rceil$.
\end{thm}

Notice that Theorem~\ref{thm:fractionalKq} is only tight for $k=2$.
The upper bound on $\phi^{\ast,k}(t,n)$ given in Theorem~\ref{thm:fractionalKq} is sufficient for our purpose, that is, to prove Theorem~\ref{thm:hajnalszemeredi} and Theorem~\ref{thm:partite}.
In addition, we also obtain the following result.

\begin{thm} \label{thm:main2}
Let $2 \le k \le t$ be integers.
Then, given any $\varepsilon, \gamma >0$, there exists an integer $n_0$ such that every $k$-graph $H$ of order $n > n_0$ with 
\begin{align}
{\delta}_{k-1}(H) \ge t \phi^{\ast,k}(t,\lceil n/t \rceil) + \gamma n \nonumber
\end{align}
contains a $K_t^k$-matching $\mathcal{T}$ covering all but at most $\varepsilon n$ vertices. 
\end{thm}

Together with Theorem~\ref{thm:fractionalKq}, we obtain the following corollary for general $k$-graphs.

\begin{cor} \label{cor:main2}
Let $3 \le k \le t$ be integers.
Then, given any $\varepsilon, \gamma >0$, there exists an integer $n_0$ such that every $k$-graph $H$ of order $n > n_0$ with
\begin{align}
{\delta}_{k-1}(H) \ge \left( 1-\binom{t-1}{k-1}^{-1} + \gamma \right) n \nonumber
\end{align}
contains a $K_t^k$-matching $\mathcal{T}$ covering all but at most $\varepsilon n$ vertices. 
\end{cor}
Observe that Corollary~\ref{cor:main2} is a stronger statement than Lemma~6.1 in~\cite{lo2011f}.
Thus, by replacing Lemma~6.1 in~\cite{lo2011f} with Theorem~\ref{thm:main2}, we improve the bounds of Theorem~1.4 in~\cite{lo2011f}.

In the next section, we prove Theorem~\ref{thm:fractionalKq}.
Theorem~\ref{thm:hajnalszemeredi} and Theorem~\ref{thm:partite} are proved simultaneously in Section~\ref{sec:proof}.
Finally, Theorem~\ref{thm:main2} is proved in Section~\ref{sec:proofofthm:main}.

\section{Perfect fractional $K_t^k$-matchings} \label{sec:fractional}

In this section we are going to prove Theorem~\ref{thm:fractionalKq}.
We require Farkas Lemma.

\begin{lma}[Farkas Lemma (see \cite{MR859549} P.257)] \label{lma:Farkas}
A system of equations $y A = b $, $y \ge 0$ is solvable if and only if the system $A x \ge 0 $, $bx <0$ is unsolvable.
\end{lma}

First we prove the lower bounds on $\phi^{\ast,k}(t,n)$.

\begin{prp} \label{prp:fractionalKqlower}
Let $2 \le k \le t$ and $n \ge 1$ be integers.
There exists a $t$-partite $k$-graph $H$ with each class of size~$n$ with $\widetilde{\delta}_{k-1}(H) = \lceil (t-k+1)n/t \rceil -1$ without a perfect fractional $K_t^k$-matching.
\end{prp}

\begin{proof}
We fix $t$, $k$ and $n$.
Let $V_1$, \dots, $V_t$ be disjoint vertex sets each of size $n$.
For $i \in [t]$, fix a $(\lceil (t-k+1)n/t \rceil -1)$-set $W_i \subset V_i$.
Define $H$ to be the $t$-partite $k$-graph on vertex classes $V_1$, \dots, $V_t$ such that every edge in $H$ meets $W_i$ for some~$i$.
Clearly, $\widetilde{\delta}_{k-1}(H) = \lceil (t-k+1)n/t \rceil -1$.
Thus, it suffices to show that $H$ does not contain a perfect fractional $K_t^k$-matching.
Let $A$ be the matrix of $H$ with rows representing the $K_t^k(H)$ and columns representing the vertices of $H$ such that $A_{T,v} =1$ if and only if $v \in T$
for $T \in \mathcal{K}_t^k(H)$ and $v \in V$.
By Farkas Lemma, Lemma~\ref{lma:Farkas}, taking $y = (w(T): T \in \mathcal{K}_t^k(H))$ and $b=(1,\dots,1)$, there is no perfect fractional $K_t^k$-matching in $H$ if and only if there is a weighting function $w: V \rightarrow \mathbb{R}$ such that 
\begin{align}
	\forall T \in \mathcal{K}_t^k(H) \  \sum_{v \in T} w(v) \ge 0 \textrm{ and } \sum_{v \in V} w(v) <0. \label{eqn:Farkas}
\end{align}
Set $w(v) = (k-1)/(t-k+1)$ if $ v \in  \bigcup_{i \in [t]} W_i$ and $w(v) = -1$ otherwise.
Clearly,
\begin{align*}
\sum w(v) = \frac{k-1}{t-k+1} t \left( \left\lceil \frac{(t-k+1)n}t \right\rceil -1 \right) - t\left( n- \left\lceil \frac{(t-k+1)n}t \right\rceil +1 \right) < 0.
\end{align*}
For $T \in \mathcal{K}_t^k(H)$, $T$ contains at least $t-k+1$ vertices in 
$ \bigcup_{i \in [t]} W_i$ and so $\sum_{v \in T} w(v) \ge 0$.
Thus, $w$ satisfies~\eqref{eqn:Farkas}, so $H$ does not contain a perfect fractional $K_t^k$-matching.
\end{proof}

\begin{proof}[Proof of Theorem~\ref{thm:fractionalKq}]
By Proposition~\ref{prp:fractionalKqlower}, it is sufficient to prove the upper bound on $\phi^{\ast,k}(t,n)$.
Fix $k$, $t$ and $n$.
Suppose the contrary that there exists a $t$-partite $k$-graph $H$ with each class of size~$n$ and 
\begin{align*}
\widetilde{\delta}_{k-1}(H) \ge \widetilde{\delta}
\end{align*}
that does not contain a perfect fractional $K_t^k$-matching, where $\widetilde{\delta}$ is the upper bound on $\phi^{\ast,k}(t,n)$ stated in the theorem.
By a similar argument as in the proof of Proposition~\ref{prp:fractionalKqlower}, there is a weighting function $w: V \rightarrow \mathbb{R}$ satisfying~\eqref{eqn:Farkas}.
Let $V_1$, \dots, $V_t$ be the vertex classes of $H$ with $V_i = \{v_{i,1}, \dots, v_{i,n}\}$ for $i \in [t]$.
We identify the $t$-tuple $(j_1, \dots, j_t) \in [n]^t$ with the $[t]$-legal set $\{v_{1,j_1}, \dots, v_{t,j_t} \}$ and write $w(j_1, \dots, j_t)$ to mean $\sum_{i \in [t]} w(v_{i,j_i})$.
Without loss of generality we may assume that for $i \in [t]$, $(w(v_{i,j}))_{j \in [n]}$ is a decreasing sequence, i.e. $w(v_{i,j}) \ge w(v_{i,j'})$ for $1 \le j < j' \le n$.
By considering the vertex weighting $w'$ such that 
\begin{align*}
	w'(v) = 
\begin{cases}
w(v) + \varepsilon & \textrm{if $v \in V_{i}$,}\\
w(v) - \varepsilon & \textrm{if $v \in V_{i'}$,}\\
w(v) & \textrm{otherwise,}
\end{cases}
\end{align*}
with $\varepsilon >0$, we may assume that $w(v_{i,n}) = w(v_{i',n})$ for all $i,i' \in [t]$.
By~\eqref{eqn:Farkas}, $w(v_{i,n})$ is negative as $w(v_{i,j}) \ge w(v_{i,n}) = w(v_{i',n})$ for all $j \in [n]$ and $i,i' \in [t]$.
Thus, by multiplying through by a suitable constant we may assume that  $w(v_{i,n}) =-1$ for all $i \in [t]$.
We further assume that $w(v) \le t-1$ for all $v \in V$, because \eqref{eqn:Farkas} still holds after we replace $w(v)$ with $\min \{ w(v), t-1 \}$.
Finally, we apply the linear transformation $(w(v)+1)/t$ for $v \in V$, which scales $w$ so that it now lies in the interval~$[0,1]$ and $w$ satisfies the following inequalities
\begin{align}
	\forall T \in \mathcal{K}_t^k(H) \  \sum_{v \in T} w(v) \ge 1 \textrm{ and } \sum_{v \in V} w(v) <n. \label{eqn:Farkas1}
\end{align}

For $j \in [t]$, set $r(j) = n - \binom{j-1}{k-1} (n- \widetilde{\delta})$.
Given a $J$-legal set $T \in \mathcal{K}_j^k(H)$ with $J \in \binom{[t]}{j}$ and $j < k$, for each $i \in [t] \backslash J$ there are at least $r(j+1)$ vertices $v \in V_i$ such that $T \cup v$ forms a $K^k_{j+1}$.
Note that $r(j) = n$ for $j \in [k-1]$ and $r(k) = \widetilde{\delta}$.
By the definition of $\widetilde{\delta}$, we know that $r(t) \ge 1$. 
Hence, we can find a $K_t^k$ $(j_1,j_2,\dots, j_t)$ with $j_i \ge r(i)$ for $i \in [t]$.

Recall that for $i \in [t]$ and $1 \le j < j' \le n$, $w(v_{i,j}) \ge w(v_{i,j'})$.
Therefore,
\begin{align}
\sum_{i \in [t]} w(v_{i,r(i)}) = w(r(1),r(2), \dots, r(t) ) \ge w(j_1,j_2,\dots, j_t) \ge 1 \nonumber
\end{align}
by~\eqref{eqn:Farkas1}.
By a similar argument, for any permutation $\sigma$ of~$[t]$ we have
\begin{align}
\sum_{i \in [t]} w(v_{i,r(\sigma(i))}) \ge 1. \nonumber
\end{align}
Setting $\sigma = (1,2,\dots,t)$, we have
\begin{align}
\sum_{i \in [t]} \sum_{j \in [t]} w(v_{i,r(j)}) = \sum_{j \in [t]} \sum_{i \in [t]}  w(v_{i,r(\sigma^j(i))})\ge t. \label{eqn:sumw1}
\end{align}
Observe that $w(v_{i,r(j)}) \le w(v_{i,r(j+1)})$ for $i \in [t]$ and $j \in [t-1]$.
Since $r(j) = n$ for $j \in [k-1]$ and $w(v_{i,n}) = 0$ for $i \in [t]$,
\begin{align}
\sum_{i \in [t]} w(v_{i,r(t)}) = & \frac{1}{t-k+1}  \sum_{i \in [t]}\left(  \sum_{j \in[k-1]} w(v_{i,r(j)}) + (t-k+1) w(v_{i,r(t)})  \right) \nonumber \\
\ge & \frac{1}{t-k+1}  \sum_{i \in [t]}\sum_{j \in [t]} w(v_{i,r(j)}) \ge  \frac{t}{t-k+1}, \label{eqn:sumw2}
\end{align}
where the last inequality is due to~\eqref{eqn:sumw1}.

\begin{clm} \label{clm:k}
\begin{align}
\sum_{i \in [t]}  \left( \sum_{ j \in [t-1]} (r(j)-r(j+1)) w(v_{i,r(j)}) + \frac{r(k) - r(t)}{t-k}  w(v_{i,r(t)}) \right) \ge \frac{t(r(k)-r(t))}{t-k} . \nonumber
\end{align}
\end{clm}

\begin{proof}[Proof of claim]
Consider the multiset $A$ containing $(t-k)(r(j)-r(j+1))$ copies of $j$ 
for $ k \le j \le t-1$ and $r(k) - r(t)$ copies of $t$.
In order to prove the claim, (by multiplying though by $(t-k)$), it is enough to show that 
\begin{align*}
	\sum_{i \in [t]} \sum_{j \in A} w(v_{i,r(j)}) \ge t(r(k)-r(t)).
\end{align*}
First note that 
\begin{align*}
 \sum_{ k \le j \le t-1} (r(j)-r(j+1)) = r(k) - r(t),
\end{align*}
so the number of elements $j$ (with multiplicity) in $A$ with $ k \le j \le t-1$ is exactly $(t-k) ( r(k) - r(t))$.
Note that $r(j)-r(j+1) = \binom{j-1}{k-2}(n-\widetilde{\delta})$.
Hence, for $k \le j < j' \le t-1$, there are more copies of $j'$ than copies of $j$ in $A$.
Recall that $A$ contains precisely $r(k) - r(t)$ copies of~$t$.
It follows that we can replace some elements by smaller elements to obtain a multiset $A'$ containing each of $k, \dots, t$ exactly $r(k)-r(t)$ times.
Since $w(v_{i,r(j)})$ is increasing in $j$ and $w(v_{i,r(j)}) = 0$ for $j \in [k-1]$, it follows that 
\begin{align*}
	\sum_{i \in [t]} \sum_{j \in A} w(v_{i,r(j)}) & \ge \sum_{i \in [t]} \sum_{j \in A'} w(v_{i,r(j)}) 
= (r(t) - r(k)) \sum_{i \in [t]} \sum_{k \le j \le t} w_{v_{i,r(j)}} \\
& = (r(t) - r(k)) \sum_{i \in [t]} \sum_{j \in [t]} w(v_{i,r(j)})
\ge t(r(t) - r(k))
\end{align*}
as required, where the last inequality is due to~\eqref{eqn:sumw1}.
\end{proof}

Recall that $r(k) = \widetilde{\delta}$ and $r(1) = n$.
Since $w(v_{i,j'})$ is decreasing in $j'$, $w(v_{i,j'}) \ge w(v_{i,r(j)})$ for $ r(j+1)<  j'  \le r(j)$ and $j \in [t]$, where we take $r(t+1) = 0$.
Hence, 
\begin{align}
\sum_{i \in [t]} \sum_{j \in [n]} w(v_{i,j}) 
\ge 
\sum_{i \in [t]}  \left( \sum_{j \in [t-1]} (r(j)- r(j+1)) w(v_{i,r(j)}) + r(t) w(v_{i,r(t)}) \right)  \nonumber.
\end{align}
By Claim~\ref{clm:k} and \eqref{eqn:sumw2}, this is at least
\begin{align}
	 & \frac{t(r(k) - r(t))}{t-k}  + \sum_{i \in [t]}  \left( r(t) - \frac{r(k) - r(t)}{t-k}  \right)w(v_{i,r(t)}) \nonumber \\
\ge & \frac{t(r(k) - r(t))}{t-k} +   \left( r(t) - \frac{r(k) - r(t)}{t-k}  \right)\frac{t}{t-k+1} 
\nonumber \\
	=  & \frac{tr(k)}{t-k+1}  =  \frac{t\widetilde{\delta}}{t-k+1}   \ge n
\nonumber 
\end{align}
contradicting~\eqref{eqn:Farkas1}.
The proof of Theorem~\ref{thm:fractionalKq} is completed.
\end{proof}

Note that the inequality above suggests that for $k \ge 3$, we would have $\phi^{\ast,k}(t,n) = \widetilde{\delta} \le \lceil (t-k+1) n/ t \rceil$.
However, our proof requires that $1 \le r(t) = n - \binom{t-1}{k-1}(n - \widetilde{\delta}) $ implying that $\widetilde{\delta} \ge \left( 1- \binom{t-1}{k-1}^{-1} \right) n+1$.

\section{Proof of Theorem~\ref{thm:hajnalszemeredi} and Theorem~\ref{thm:partite}} \label{sec:proof}

First we need the following simple proposition.

\begin{prp} \label{prp:close}
Let $\gamma >0$.
Let $H$ be a balanced $t$-partite $k$-graph with partition classes $V_1$, \dots, $V_t$, each of size $n$ with 
\begin{align}
\widetilde{\delta}_{k-1}(H) \ge \left( 1 - \left( \binom{t-2}{k-1} + 2 \binom{t-2}{k-2} \right)^{-1} + \gamma\right)n. \nonumber
\end{align}
Then, for $i \in [t]$ and distinct vertices $u, v \in V_i$, there are at least $(\gamma n)^{t-1}$ legal $[t]\backslash i$-sets $T$ such that $T \cup u$ and $T \cup v$ span copies of $K_t^k$ in~$H$.
\end{prp}

\begin{proof}
Let $u,v \in V_1$.
For $2 \le i \le t$, we pick $w_i \in V_i$ such that $w_i \in N(T)$ for all legal $(k-1)$-sets $T \subset \{ u,v, w_2, \dots, w_{i-1}\}$.
By the definition of $\widetilde{\delta}_{k-1}(H)$, there are at least $\gamma n$ choices for each~$w_i$.
The proposition easily follows.
\end{proof}

Using Proposition~\ref{prp:close}, we obtain an absorption lemma. 
Its proof can be easily obtained by modifying the proof of Lemma~4.2 in~\cite{lo2011perfect}.
For the sake of completeness, it is included in Appendix~A.

\begin{lma}[Absorption lemma] \label{lma:absorptionlemma}
Let $2 \le k < t$ be integers and let $\gamma >0$.
Then, there is an integer $n_0$ satisfying the following: for each balanced $t$-partite $k$-graph $H$ with each class of size $n \ge n_0$ and 
\begin{align}
\widetilde{\delta}_{k-1}(H) \ge \left( 1 - \left( \binom{t-2}{k-1} + 2 \binom{t-2}{k-2} \right)^{-1} + \gamma\right)n, \nonumber
\end{align}
there exists a balanced vertex subset $U \subset V(H)$ of size $|U| \le \gamma^{t(t-1)} n/(t^2 2^{t+2})$ such that there exists a perfect $K_t^k$-matching in $H[U \cup W]$ for every balanced vertex subset $W \subset V \backslash U$ of size $|W|  \le \gamma^{2t(t-1)} n/(t^2 2^{2t+5})$.
\end{lma}

Our next task is to find a large $K_t^k$-matching in $H$ covering all but at most $\varepsilon n$ vertices, which requires a theorem of Frankl and R{\"o}dl~\cite{MR829351} and Chernoff's inequality.
The proof of Lemma~\ref{lma:approximate} is based on Claim~4.1 in~\cite{alon2011large}.
For constants $a,b,c >0$, write $a = b \pm c$ for $b-c \le a \le b+c$.

\begin{thm}[Frankl and R{\"o}dl~\cite{MR829351}] \label{thm:FranklRodl}
For all $t, \varepsilon \ge 0$ and $a >3$, there exists $\tau = \tau(\varepsilon)$, $D = D(n)$, and $n_0 = n_0 (\tau)$ such that if $n \ge n_0$ and $H$ is a $t$-graph of order~$n$ satisfying
\begin{enumerate}
\item $\deg^H(v) = (1 \pm \tau) D$ for all $v \in V$, and 
\item $\Delta_2(H) = \max_{T \in \binom{V(H)}{2}} \deg^{H}(T)< D/ (\log n )^a$
\end{enumerate}
then $H$contains a matching $M$ covering all but at most $\varepsilon n$ vertices. 
\end{thm}

\begin{lma}[Chernoff's inequality (see e.g.~\cite{MR1885388})] \label{lma:Chernoff}
Let $X \sim Bin(n,p)$.
Then, for $0 < \lambda \le np$
\begin{align*}
\mathbb{P}(|X-np| \ge \lambda) \le 2 \exp\left( -\frac{\lambda^2}{4 np}\right) \textrm{ and }
\mathbb{P}(X \le np- \lambda) \le \exp\left( -\frac{\lambda^2}{4 np}\right).
\end{align*}
\end{lma}

\begin{lma} \label{lma:approximate}
Let $2 \le k \le t$ be integers.
Then, for any given $\varepsilon, \gamma >0$, there exists an integer $n_0$ such that every $t$-partite $k$-graph $H$ with partition classes $V_1$, \dots, $V_t$, each of size $n >n_0$, with
\begin{align}
\widetilde{\delta}_{k-1}(H) \ge \phi^{\ast,k}(t,n) + \gamma n \nonumber
\end{align}
contains a $K_t^k$-matching $\mathcal{T}$ covering all but at most $\varepsilon n$ vertices. 
\end{lma}

\begin{proof}
Fix $k$, $t$ and $\varepsilon$.
If $k=t=2$, then the lemma easily holds and so we may assume that $t \ge 3$.
Write $\phi^{\ast} = \phi^{\ast,k}(t,n)/n$.
We assume that $n$ is sufficiently large throughout the proof.
Let $H$ be a balanced $t$-partite $k$-graph $H$ with partition classes $V_1$, \dots, $V_t$, each of size $n$, with $\widetilde{\delta}_{k-1}(H) \ge (\phi^{\ast} + \gamma) n$.
Our aim is to define a $t$-graph $H^{\ast}$ on vertex set $V(H)$ satisfying the condition of Theorem~\ref{thm:FranklRodl}, where every edge in $H^{\ast}$ corresponds to a $K_t^k$ in~$H$.
Hence, by Theorem~\ref{thm:FranklRodl}, there exists a matching $M$ covering all but at most $\varepsilon n$ vertices of $H^{\ast}$ corresponding to a $K_t^k$-matching in~$H$.

We are going to construct $H^{\ast}$ via two rounds of randomisation. 
For $i \in [t]$, let $R_i$ be a random binomial subset of $V_i$ with probability $p = n^{-0.9}$.
Let $R = (R_1, \dots, R_t)$ .
Then, by Chernoff's inequality (Lemma~\ref{lma:Chernoff})
\begin{align}
	\mathbb{P}(|R_i - n^{0.1}| \ge n^{0.075}) \le 2 \exp(-n^{0.05}/2). \label{eqn:|R_i|}
\end{align}
For each $I \in \binom{[t]}{k-1}$, each $I$-legal set $T \subset R$ and $i \in [t] \backslash I$
\begin{align*}
	\mathbb{E}(\deg^{H[R]}_i(T)) \ge (\phi^{\ast}+\gamma) n \times n^{-0.9} = (\phi^{\ast}+\gamma) n^{0.1}.
\end{align*}
Again, by Chernoff's inequality (Lemma~\ref{lma:Chernoff})
\begin{align}
	\mathbb{P}(\deg^{H[R]}_i(T) <  (\phi^{\ast}+\gamma/2) n^{0.1}) \le \exp(-\gamma^2 n^{0.1}/(16(\phi^{\ast} + \gamma))) = e^{ - \Omega ( n^{0.1} )}. \label{eqn:deg[T]}
\end{align}
Let $m =  n^{0.1} - n^{0.075}$.
Let $R_i'$ be a randomly chosen $m$-set in $R_i$ and let $R' = (R'_1, \dots, R'_t)$.
By \eqref{eqn:|R_i|} and~\eqref{eqn:deg[T]}, we have with probability $1 - e^{- \Omega(n^{0.05})}$
\begin{align*}
	\widetilde{\delta}_{k-1}(H[R'])  \ge (c+\gamma/2) n^{0.1} - 2 n^{0.075}\ge (c+\gamma/4) m.
\end{align*}
Since $R'_i$ is chosen randomly from $R_i$, which is also chosen randomly, a given element is chosen in $R'_i$ with probability $m/n = n^{-0.9}- n^{-0.925}$ minus an exponentially small correction term.
Hence we may assume that for $v \in V$
\begin{align*}
 n^{-0.9} \ge \mathbb{P}(v \in R') \ge (1-2n^{-0.025}) n^{-0.9}.
\end{align*}
Now, we take $n^{1.1}$ independent copies of $R'$ and denote them by $R'(1)$, $R'(2)$, \dots, $R'(n^{1.1})$.
For a subset of vertices $S \subset V$, let 
\begin{align}
	Y_S = |\{i : S \subset R'(i)\}|. \nonumber
\end{align}
Since the probability that a particular $R^i$ (not $R'(i)$) contains $S$ is $n^{-0.9 n}$, $ \mathbb{E}(Y_S) \le n^{1.1 - 0.9|S|}$.
With probability at least $1- 2 \exp (-9n^{1.5}/2) $ by Lemma~\ref{lma:Chernoff}, $Y_v =  n^{0.2} \pm  3n^{0.175}$ for every $v \in V$,
where recall that $y = x \pm c$ means $x - c \le y \le x +c$.
Let $Z_2 = |\{ S \in \binom{V}{2} : Y_S \ge 3\}|$ and observe that
\begin{align}
	\mathbb{E} (Z_2) < n^2 \left( n^{1.1}\right)^3 \left( n^{-0.9}\right)^6 = n^{-0.1}. \nonumber
\end{align}
Let $Z_3 = |\{ S \in \binom{V}{3} : Y_S \ge 2\}|$ and observe that
\begin{align}
	\mathbb{E} (Z_3) < n^3 \left( n^{1.1}\right)^2 \left( n^{-0.9}\right)^{6} = n^{-0.2}. \nonumber
\end{align}
The latter implies that every 3-set $S \in \binom{V}3$ lies in at most one $R'(i)$ with high probability.
In summary, there exist $n^{1.1}$ vertex sets $R'(1)$, \dots, $R'(n^{1.1})$ such that
\begin{itemize}
\item[$(i)$] for every $v \in V$, $Y_v =  n^{0.2} \pm  3n^{0.175}$,
\item[$(ii)$] every 2-set $S \in \binom{V}2$ is in at most two sets $R'(i)$,
\item[$(iii)$] every 3-set $S \in \binom{V}3$ is in at most one set $R'(i)$,
\item[$(iv)$] for $i \in [n^{1.1}]$, $R'(i) = (R'_1,\dots R'_t)$ with $R'_j \subset V_j$ and $|R'_j| = m$ for $j \in [t]$,
\item[$(v)$] for $i \in [n^{1.1}]$, $\widetilde{\delta}_{k-1}(H[R'(i)]) \ge (\phi^{\ast}+\gamma/4)m$.
\end{itemize}
Fix one such sequence $R'(1)$, \dots, $R'(n^{1.1})$.

By $(v)$ and the definition of $\phi^{\ast}$, there exists a fractional perfect $K_t^k$-matching $w^{i}$ in $H[R'(i)]$ for $i \in [n^{1,1}]$.
Now we conduct our second round of random process by defining a random $t$-graph $H^{\ast}$ on vertex classes $V$ such that each $[t]$-legal set $T$ is randomly independently chosen with
\begin{align}
	\mathbb{P} (T \in H^{\ast}) = 
\begin{cases}
	w^{i_T}(T)	& \textrm{if $T \in \mathcal{K}_t^k(H[R'(i_T)])$ for some $i_T \in [t]$,}\\
	0	&	\textrm{otherwise}. 
\end{cases} \nonumber
\end{align}
Note that $i_T$ is unique by~$(iii)$ (as $t \ge 3$) and so $H^{\ast}$ is well defined. 
For $v \in V$, let $I_v = \{ i : v \in R'(i)\}$ and so $|I_v| = Y_v = n^{0.2} \pm  3n^{0.175}$ by~$(i)$.
For every $v \in V$, let $E^i_v$ be the set of $K_t^k$ in $H[R'(i)]$ containing~$v$.
Thus, for $v \in V$, $\deg^{H^{\ast}} (v)$ is a generalised binomial random variable with expectation
\begin{align*}
	\mathbb{E}(\deg^{H^{\ast}} (v)) = \sum_{i \in I_v} \sum_{T \in E^i_v} w^i(T) = |I_v| = n^{0.2} \pm  3n^{0.175}.
\end{align*}
Similarly, for every 2-set $\{u,v\}$,
\begin{align*}
	\mathbb{E}(\deg^{H^{\ast}} (u,v)) = \sum_{i \in I_v \cap I_u} \sum_{T \in E^i_v \cap E^i_u} w^i(T) \le |I_v \cap I_u|  \le 2,
\end{align*}
by $(ii)$.
Hence, again by Chernoff's inequality, Lemma~\ref{lma:Chernoff}, we may assume that for every $v \in V$ and every 2-set $\{u,v\}$
\begin{align*}
	\deg^{H^{\ast}} (v) = n^{0.2} \pm  4n^{0.2-\varepsilon}, \qquad
	\deg^{H^{\ast}} (u,v) < n^{0.1}.
\end{align*}
Thus, $H^{\ast}$ satisfies the hypothesis of Theorem~\ref{thm:FranklRodl} and the proof is completed.
\end{proof}

Next we prove Theorem~\ref{thm:hajnalszemeredi} and Theorem~\ref{thm:partite}.
\begin{proof}[Proof of Theorem~\ref{thm:hajnalszemeredi} and Theorem~\ref{thm:partite}]
Fix $k$ and $t$ and $\gamma > 0$.
Let 
\begin{align*}
d = \begin{cases}
(t-1)n/t & \textrm{if $k=2$}\\
\left( 1 - \left( \binom{t-1}{k-1} + 2 \binom{t-2}{k-2} \right)^{-1} \right)n & \textrm{if $k\ge 3$}.
\end{cases}
\end{align*}
Note that $d \ge \phi^{\ast,k}(t,n)$ by Theorem~\ref{thm:fractionalKq}.
Let $H$ be a $t$-partite $k$-graph with vertex classes $V_1$, \dots, $V_t$ each of size $n \ge n_0$ and $\widetilde{\delta}_{k-1}(H) \ge d+\gamma n$.
We are going to show that $H$ contains a perfect $K_t^k$-matching.
Throughout this proof, $n_0$ is assumed to be sufficiently large.
By Lemma~\ref{lma:absorptionlemma}, there exists a balanced vertex set $U$ in $V$ of size $|U| \le \gamma^{t(t-1)} n/(t^2 2^{t+2})$ such that there exists a perfect $K_t^k$-matching in $H[U \cup W]$ for every balanced vertex subset $W \subset V \backslash U$ of size $|W|  \le \gamma^{2t(t-1)} n/(t^2 2^{2t+5})$.
Set $H' = H[V \backslash U]$ and note that $\widetilde{\delta}_{k-1}(H') \ge d+\gamma n/2 \ge (\phi^{\ast,k}(t,n) + \gamma/2) n$.
By Lemma~\ref{lma:approximate}, there exists a $K_t^k$-matching $\mathcal{T}$ in $H'$ covering all but at most $\varepsilon n$ vertices of $H'$, where $\varepsilon =  \gamma^{2t(t-1)} /(t^2 2^{2t+5})$.
Let $W = V(H') \backslash V(\mathcal{T})$, so $W$ is balanced.
Since $H[U \cup W]$ contains a perfect $K_t^k$-matching~$\mathcal{T}'$ by the choice of~$U$, $\mathcal{T} \cup \mathcal{T}'$ is a perfect $K_t^k$-matching in~$H$.
\end{proof}

\section{Proof of Theorem~\ref{thm:main2}} \label{sec:proofofthm:main}

Note that together Lemma~\ref{lma:approximate} and the lemma below imply Theorem~\ref{thm:main2}.
Hence all that remains is to prove Lemma~\ref{lma:makepartite}.

\begin{lma} \label{lma:makepartite}
For integers $t \ge k \ge 2$, there exists $n_0$ such that the following holds.
Suppose that $H$ is a $k$-graph with $n \ge n_0$ vertices with $t |n$.
Then there exists a partition $V_1$, \dots, $V_t$ of $V(H)$ into sets of size $n/t$ such that for every $l \in [k-1]$, every $I \in \binom{[t]}{l}$, every legal $I$-set $T$ and $J \in \binom{[t] \backslash I}{k-l}$, we have 
\begin{align*}
 \frac{t^{k-l}}{(k-l)!}\deg^{H'}_J(T)  \ge \deg^H(T) - 2 (t \ln n)^{1/2} n^{k-l-1/2},
\end{align*}
where $H'$ is the induced $t$-partite $k$-subgraph of $H$ with vertex classes $V_1$, \dots, $V_t$.
\end{lma}

\begin{proof}
First set $m= k-l$ and let $U_1$, \dots, $U_t$ be a random partition of $V$, where each vertex appears in vertex class $U_j$ independently with probability $1/t$.
For a fixed $l$-set $T = \{v_1, \dots, v_l\}$, let $N^H(T)$ be the link hypergraph of $T$. 
Thus, $N^H(T)$ is an $m$-graph with $\deg^H(T)$ edges. 
We decompose $N^H(T)$ into $i_0 \le m n^{m-1}$ nonempty pairwise edge disjoint matchings, which we denote by $M_1$, \dots, $M_{i_0}$.
To see that this is possible consider the auxiliary graph $G$ with $V(G) = E(N^H(T))$, in which for $A, B \in N^H(T)$ $A$ and $B$ are joined in $G$ if and only if $A \cap B \ne \emptyset$.
Since $G$ has maximum degree at most $m \binom{n-1}{m-1}$, $G$ can be properly coloured using at most $m n^{m-1}$ colours, where each colour class corresponds to a matching.

For every edge $E \in N^H(T)$, and every index set $J \in \binom{[t]}{m}$, we say that $E$ is \emph{$J$-good}, if $E$ is $J$-legal with respect to $U_1$, \dots, $U_t$.
Since the partition $U_1$, \dots, $U_t$ was chosen randomly, we have for fixed $J \in \binom{[t]}{m}$ $$ \mathbb{P}(E \textrm{ is $J$-good}) = m!t^{-m} .$$
Thus, for $X_{i,J} = X_{i,J}(T) = |\{ E \in M_i : E \textrm{ is $J$-good} \} |$ we have
\begin{align*}
	\mu_{i,J} = \mu_{i,J}(T) = \mathbb{E}(X_{i,J}) = \frac{m!}{t^m} |M_i|.
\end{align*}
Now call a matching $M_i$ \emph{bad} (with respect to $U_1$, \dots, $U_t$) if there exists a set $J \in \binom{[t]}{m}$ such that
\begin{align*}
	X_{i,J} \le \left(1 - \left( \frac{2(2k-1) \ln n }{\mu_{i,J}}\right)^{1/2} \right) \mu_{i,J}
\end{align*}
and call $T$ a \emph{bad set} if there is at least one bad $M_i = M_i(T)$.
Otherwise call $T$ a \emph{good set}. 
For a fixed $M_i$ the events `$E$ is $J$-good' with $E \in M_i$ are jointly independent, hence by Chernoff's inequality, Lemma~\ref{lma:Chernoff},
\begin{align}
	\mathbb{P}(M_i \textrm{ is bad}) \le \binom{t}{m} \exp(-(2k-1) \ln n) = \binom{t}m n^{-2k+1}. \nonumber
\end{align}
Recall that $i_0 \le m n^{m-1}$ and $m \le k-1$, we have 
$$\mathbb{P}(T \textrm{ is bad}) \le i_0 \binom{t}m n^{-2k+1} \le n^{-k}$$
and by summing over all $l$-sets $T$ we obtain that 
$$\mathbb{P}(\textrm{there exists a bad $l$-set}) \le n^{-1}.$$
Moreover, Chernoff's inequality, Lemma~\ref{lma:Chernoff}, yields
\begin{align*}
\mathbb{P}( |U_{j}| \ge  n/t  + n^{1/2} (\ln n )^{1/4}/t)\le \exp(-(\ln n)^{1/2} /4t).
\end{align*}
Thus with positive probability there is a partition $U_1$, \dots, $U_t$ such that all $l$-sets $T$ are good and 
\begin{align*}
|U_{j}| \le n/t  + n^{1/2} (\ln n )^{1/4}/t \textrm{ for all } j \in [t].
\end{align*}
Consequently, by redistributing at most $n^{1/2} (\ln n )^{1/4}$ vertices of the partition $U_1$, \dots, $U_t$ we obtain an equipartition $V_1$, \dots, $V_t$ with 
\begin{align*}
|V_j| = n/t \textrm{ and } |U_{j} \backslash V_j| \le n^{1/2} (\ln n )^{1/4}/t \textrm{ for all } j \in [t].
\end{align*}
Let $H'$ be the induced $t$-partite $k$-subgraph with vertex classes $V_1$, \dots, $V_t$.
Note that for an $l$-set $I \in \binom{[t]}{l}$, a $I$-legal set $T$ and an $m$-set $J \in \binom{[t] \backslash I}{m}$,
\begin{align*}
\deg_J^{H'} (T) \ge & \sum_{i \in [i_0]} \left(1 - \left( \frac{2 (2k-1) \ln n }{\mu_{i,J}}\right)\right) \mu_{i,J} - m \frac{n^{1/2} (\ln n )^{1/4}}t n^{m-1}\\
\ge & \frac{m!}{t^m} \deg^H_J(T) - \left( 2 (2k-1) \ln n \right)^{1/2} \sum_{i \in [i_0]} \mu_{i,J}^{1/2} - m \frac{n^{1/2} (\ln n )^{1/4}}t n^{m-1}.
\end{align*}
By the Cauchy-Schwarz inequality, we obtain that 
\begin{align*}
 \sum_{i \in [i_0]} \mu_{i,J}^{1/2} \le \left( i_0 \sum_{i \in [i_0]} \mu_{i,J} \right)^{1/2} \le \left( mn^{m-1} \frac{m!}{t^m} \binom{n}m \right)^{1/2} \le n^{m - 1/2}
\end{align*}
Therefore, 
\begin{align*}
\deg^{H}_J (T) \ge & \frac{m!}{t^m} \deg^H_J(T) -  2(k \ln n)^{1/2}n^{m-1/2},
\end{align*}
as required.
\end{proof}

\section{Acknowledgment}
The authors would like to thank the anonymous referee for the helpful comments.

\appendix
\section{Proof of Lemma~\ref{lma:absorptionlemma}}

\begin{proof}
Throughout the proof we may assume that $n_0$ is chosen sufficiently large.
Let $H$ be a balanced $t$-partite $k$-graph with partition classes $V_1$, \dots, $V_t$ each of size $n$ and $\widetilde{\delta}_{k-1}(H) \ge \widetilde{\delta}$, where $\widetilde{\delta}$ is the lower bound on $\widetilde{\delta}_{k-1}(H)$ stated in the lemma.
Let $H'$ be the $t$-partite $t$-graph on $V_1$, \dots, $V_t$ in which $v_1 v_2 \dots v_t \in E(H')$ if and only if $v_1 v_2 \dots v_t$ is a $K_t^k$ in $H$.
Furthermore set $m=t(t-1)$ and call a balanced $m$-set $A$ an \emph{absorbing} $m$-set for a balanced $t$-set $T$ if $A$ spans a matching of size $t-1$ in $H'$ and $A \cup T$ spans a matching of size $t$ in $H'$, in other words, $A \cap T = \emptyset$ and both $H'[A]$ and $H'[A \cup T]$ contain a perfect matching.
Denote by $\mathcal{L}(T)$ the set of all absorbing $m$-sets for $T$.
Next, we show that for every balanced $t$-set $T$, there are many absorbing $m$-sets for $T$.

\begin{clm} \label{clm:numberofabsorbingm-set}
For every balanced $t$-set $T$, $|\mathcal{L}(T)| \ge \gamma^{m} \binom{n}{t-1}^t/2^t $.
\end{clm}

\begin{proof}
Let $T= \{v_1, \dots, v_t\}$ be fixed with $v_i \in V_i$ for $i \in[t]$.
By Proposition~\ref{prp:close}, it is easy to see that there exist at least $(\gamma n)^{t-1}$ edges in $H'$ containing $v_1$.
Since $n_0$ was chosen large enough, there are at most $(t-1)n^{t-2} \le (\gamma n)^{t-1}/2$ edges in $H'$ which contain $v_1$ and $v_j$ for some $2 \le j \le t$.
Fix an edge $v_1u_2\dots u_t$ in $H'$ with $u_j \in V_j \backslash \{v_j\}$ for $2 \le j \le t$.
Set $U_1 = \{ u_2, \dots, u_t\}$ and $W_0 = T$.
For each $2 \le j \le t$, suppose we succeed to choose a $(t-1)$-set $U_j$ such that $U_j$ is disjoint from $W_{j-1} = U_{j-1} \cup W_{j-2}$ and both $U_j \cup \{ u_j \}$ and $U_j \cup \{ v_j\}$ are edges in $H'$.
Then for a fixed $2 \le j \le t$ we call such a choice $U_j$ \emph{good}, motivated by $A = \bigcup_{j \in [t]} U_j$ being an absorbing $m$-set for $T$.

Note that in each step $2 \le j \le t$ there are precisely $t+(j-1)(t-1)$ vertices in $W_{j-1}$.
More specifically, for $i \in [t]$, there are at most $j \le t$ vertices in $V_i \cap W_{j-1}$.
Thus, the number of edges in $H'$ intersecting $u_j$ (or $v_j$ respectively) and at least one other vertex in $W_j$ is at most $(t-1) j n^{t-2} < t^2n^{t-2} \le (\gamma n)^{t-1}/2$.
For each $2 \le j \le t$, by Proposition~\ref{prp:close} there are at least $( \gamma n)^{t-1} - (\gamma n)^{t-1}/2  =(\gamma n)^{t-1}/2$ choices for $U_j$ and in total we obtain $(\gamma n)^{m}/2^t$ absorbing $m$-sets for $T$ with multiplicity at most $((t-1)!)^{t}$.
\end{proof}

Now, choose a family $F$ of balanced $m$-sets by selecting each of the $\binom{n}{t-1}^t$ possible balanced $m$-sets independently with probability $$p = \gamma^m n/ \left(t^3 2^{t+3} \binom{n}{t-1}^t \right).$$
Then, by Chernoff's inequality, Lemma~\ref{lma:Chernoff} with probability $1-o(1)$ as $n \rightarrow \infty$, the family $F$ satisfies the following properties:
\begin{align}
|F| \le & \gamma^m n / (t^3 2^{t+2})\label{eqn:|F|}
\intertext{and}
|\mathcal{L}(T) \cap F| \ge & \frac{\gamma^{2m} n}{t^3 2^{2t+4}} \label{eqn:L(T)}
\end{align} 
for all balanced $t$-sets $T$.
Furthermore, we can bound the expected number of intersecting $m$-sets in $F$ by
\begin{align}
	\binom{n}{t-1}^t \times t(t-1) \times \binom{n}{t-2}  \binom{n}{t-1}^{t-1} \times p^2 \le  \frac{ \gamma^{2m} n }{t^3 2^{2t+6}}
\nonumber
\end{align}
Thus, using Markov's inequality, we derive that with probability at least $1/2$
\begin{align}
	\textrm{$F$ contains at most $  \frac{ \gamma^{2m} n }{t^32^{2t+5}}$ intersecting pairs.} \label{eqn:F}
\end{align}
Hence, with positive probability the family $F$ has all properties stated in \eqref{eqn:|F|}, \eqref{eqn:L(T)} and \eqref{eqn:F}.
By deleting all the intersecting balanced $m$-sets and non-absorbing $m$-sets in such a family $F$, we get a subfamily $F'$ consisting of pairwise disjoint balanced $m$-sets, which satisfies
\begin{align}
|\mathcal{L}(T) \cap F'| \ge & \frac{ \gamma^{2m} n }{t^3 2^{2t+4}} - \frac{ \gamma^{2m} n }{t^32^{2t+5}} = \frac{ \gamma^{2m} n }{t^32^{2t+5}} \nonumber
\end{align}
for all balanced $t$-sets $T$.
Let $U = V(F')$ and so $U$ is balanced.
Moreover, $U$ is of size at most $t |V(F')| \le t |V(F)| \le \gamma^m n / (t^2 2^{t+2})$ by~\eqref{eqn:|F|}.
For a balanced set $W \subset V \backslash V(M)$ of size $|W| \le \frac{ \gamma^{2m} n }{t^2 2^{2t+5}}$, $W$ can be partition in to at most $\frac{ \gamma^{2m} n }{t^3 2^{2t+5}}$ balanced $t$-set.
Each balanced $t$-set can be successively absorbed using a different absorbing $m$-set in~$F'$, so there exists a perfect matching in $H'[U \cup W]$. 
Hence, there is a perfect $K_t^k$-matching in $H[U \cup W]$.
\end{proof}

\end{document}